\documentclass{elsarticle}

\usepackage{amssymb}
\usepackage{xcolor}
\usepackage{amsmath}
\usepackage{mathtools}
\usepackage{url}

\newtheorem{theorem}{Theorem}
\newtheorem{lemma}[theorem]{Lemma}
\newtheorem{corollary}[theorem]{Corollary}
\newtheorem{proposition}[theorem]{Proposition}
\newtheorem{remark}[theorem]{Remark}

\newtheorem{definition}[theorem]{Definition}
\newproof{proof}{Proof}

\newcommand{\ff}{\mathbb{F}}

\newcommand{\n}{\mathbb{N}}
\newcommand{\X}{\textbf{X}}
\newcommand{\spann}{\text{span}_{\n}\{\gamma_1, \gamma_2, \dots, \gamma_{m-1}\} }

\newcommand{\spannnn}{\text{span}_{\n}(\gamma_1, \dots \gamma_{k-1})}

\usepackage{mathtools}

\DeclarePairedDelimiter\floor{\lfloor}{\rfloor}

\begin{document}     
\title{A new criterion for the absolute irreducibility  of multivariate polynomials over finite fields}

\author[1]{Carlos Agrinsoni\corref{cor1}
  \fnref{fn1}}
\ead{cagrinso@pudue.edu}

\author[2]{Heeralal Janwa\fnref{fn2}}
\ead{heeralal.janwa@upr.edu}

\author[3]{Moises Delgado\fnref{fn3}}
\ead{moises.delgado@upr.edu}

 \cortext[cor1]{Corresponding author}
 \fntext[fn1]{Department of Mathematics, Purdue University, 150 N University St, West Lafayette, Indiana, 47907-2067, USA.}
 \fntext[fn2]{Department of Mathematics, University of Puerto Rico, Rio Piedras, College of Natural Sciences 17 University Ave. Ste 1701 San Juan PR, 00925-2537.}
 \fntext[fn3]{Department of Mathematics and Physics, University of Puerto Rico, Cayey, 205 Calle Antonio R. Barceló, Cayey PR, 00736.}

\begin{abstract}

A key property of an algebraic variety is whether it is absolutely irreducible, meaning that it remains irreducible over the algebraic closure of its defining field, and determining absolute irreducibility is important in algebraic geometry and its applications in coding theory, cryptography, and other fields.

Among the applications of absolute irreducibility are bounding the number of rational points via the Weil conjectures and establishing exceptional APN and permutation properties of functions over finite fields.
In this article, we present a new criterion for the absolute irreducibility of hypersurfaces defined by multivariate polynomials over finite fields. Our criterion does not require testing for irreducibility in the ground or extension fields, assuming that the leading form is square-free. We just require multivariate GCD computations and the square-free property. Since almost all polynomials are known to be square-free, our absolute irreducibility criterion is valid for almost all multivariate polynomials.

\end{abstract}

\maketitle

\section{Introduction}

One of the most important problems in algebra and algebraic geometry is to test when a polynomial is irreducible (respectively, absolutely irreducible). Developing practical criteria to test irreducibility and absolute irreducibility is essential for applications in pure and applied mathematics. For polynomials in one variable, the Eisenstein criterion is a classical method to test irreducibility \cite{eisenstein1850}. This criterion has been generalized to test the irreducibility (and, in some cases, absolute irreducibility) of multivariate polynomials using Newton polygons, prime ideals and valuations (see, for example, \cite{Cassels1986, Maclane1938, KS1997}).

The absolute irreducibility property has many applications in different areas of mathematics such as algebraic geometry, combinatorics \cite{Szhonyi1997}, coding theory \cite{stichtenoth1993, JW1993, JMW1995}, cryptography \cite{HM2011}, finite geometry \cite{HM2012, Hirschfeld1979}, function field sieve \cite{Adleman1994}, permutation polynomials \cite{LN1983} and scattered polynomials \cite{BM2019}. Deep results such as the Weil, Deligne, and Bombieri bounds on the number of rational points and exponential sums require that the underlying variety be absolutely irreducible (for singular curves, see \cite{AP1996}).

The Weil conjectures play an essential role in applications to bound the number of rational points and exponential sums. Moreover, Deligne's Theorem that gives bounds on rational points requires that the underlying variety is absolutely irreducible \cite{deligne1974}. Generalizations of Weil’s conjectures to singular curves require that the polynomial defining the singular curve be absolutely irreducible \cite{AP1996}. 
The following concrete applications are particularly relevant to our results. In coding theory and cryptography, significant advances have been made in proving the exceptional conjecture of APN by demonstrating that a certain variety, defined by the corresponding multivariate polynomial, is absolutely irreducible or contains an absolutely irreducible component \cite{HM2011}. This conjecture was formulated for monomial functions by Janwa and Wilson \cite{JW1993} with significant progress. Their investigation showed that one of the existing absolute irreducibility testing criteria was effective, and they developed new methods.  After half a decade, Janwa gave an absolute irreducibility criterion based on Bezout's Theorem and intersection multiplicities. This led to significant progress by Janwa, McGuire, and Wilson \cite{JMW1995} who mostly settled the conjecture. After two decades,  it was finally proved by Hernando and McGuire \cite{HM2012}. Aubry, McGuire, and Rodier \cite{amw2010} generalized this conjecture for arbitrary functions over finite fields and the existence of an absolutely irreducible factor for the corresponding variety.

In this article, we consider hypersurfaces defined by a multivariate polynomial. There exist only a few criteria for testing the absolute irreducibility of a polynomial. 
The method of Newton polygons generalizes the Eisenstein criterion in two different ways: 1) Eisenstein-Dumas criterion \cite{dumas1906, Wan1995}; 2) the Stepanov criterion (sometimes called the Stepanov-Schmidt criterion) \cite{ schmidt2004, stepanov1972, stepanov1974}.      A more general criterion uses 
Newton polytopes \cite{gao2001}. Other criteria are given in \cite{kaltofen1985, von1985,  HS1981, berlekamp1967, berlekamp1970, LLL1982, kaltofen1982, kaltofen1982IEEE, kaltofen1985SIAM, lenstra1985, weinberger1984, kaltofen1995}. 
However, in proving absolute irreducibility theorems, these algorithms are often not practical. Indeed, the general criterion of Gao is, in general, impractical, as the algorithm requires the indecomposability of polytopes and is known to be an NP-complete problem \cite{gao2001}.

The test of absolute irreducibility that we present here has been very effective in proving substantial parts of the APN conjecture (see \cite{AJDGold, ajd2025evencase}). Our strategy is reminiscent of the absolute irreducibility criteria developed by Janwa and Wilson \cite{JW1993}.
Janwa, McGuire, and Janwa give a highly effective method \cite{JW1993, JMW1995} using Bezout’s theorem and intersection multiplicities. This method has been used to settle the Segre-Bartocci conjecture \cite{HM2012}, and to reduce the general exceptional APN conjecture to exponents in the odd case involving only the Gold and Kasami-Welch cases. The criteria of \cite{JW1993, amw2010} have applications to the exceptional APN conjecture,   to the Segre-Bartocci conjecture \cite{HM2012}, and to the e. Aubry, McGuire, and Rodier \cite{amw2010} reduce such problems in three dimensions to two dimensions by hyperplane intersections for applications to the general exceptional APN function conjecture.

The criterion we present here is an important addition to the proof of the exceptional APN conjecture, among other applications.  One of the conditions to define algebraic geometric codes is that the underlying curve is absolutely irreducible \cite{hoholdt1998algebraic}. In finite geometry, the Segre-Bartocci conjecture was settled by proving that a multivariate polynomial contains an absolutely irreducible factor over the defining field. Absolute irreducibility is fundamental to proving that a permutation polynomial is exceptional

Our previous results are given in \cite{DJ2018, agrinsoniproms}.
We also prove results on the existence of absolutely irreducible factors over the defining field. This factor guarantees a sufficient number of rational points by the bounds of Weil, Deligne, Bombieri, and Lachaud-Ghorpade (see \cite{amw2010, JW1993}), in proving the Segre-Bartocci conjecture \cite{HM2012} and the exceptional planar conjecture \cite{BS2019}, the exceptional APN function conjecture, among others.

The article is organized as follows. In Section \ref{sect: background}, we provide a background. In Section \ref{sect: New results}, we prove the main result of this article that gives an absolute irreducibility criterion (Theorem \ref{thm: heptanomials}). In Section \ref{sect: consequences of the main theorem}, we generalize the concept of degree-gap, which helps us improve our earlier bound on the number of absolutely irreducible factors. As a consequence, we obtain a direct proof of the absolute irreducibility of a class of multivariate polynomials. 
The final section contains the conclusions.

\section{Background}\label{sect: background}
A multivariate polynomial $G(\X)$ in  $\ff(\textbf{X}):=\ff[X_1, \dots, X_n]$ is called absolutely irreducible if it remains irreducible over every extension of $\ff$. Furthermore, in the graded degree homogeneous decomposition (in the decreasing degree order),
$G(\textbf{X})$ decomposes as 
\begin{equation}\label{eqn:graded}
G(\textbf{X}) = G_d(\textbf{X}) + G_{d_1}(\X) \dots + G_{d_m}(\textbf{X}),
\end{equation}

\noindent where $G_{d_i}(\textbf{X})$ is a non-zero homogeneous polynomial of degree $d_i$, $1 \leq i \leq m$, and $d > d_1 > \cdots > d_m$. In the rest of the article, polynomials will be written in decreasing degree homogeneous forms. The homogeneous component $G_{d_m}(\X):= T_G(\textbf{X}) := T(\textbf{X})$ is called the \textbf{tangent cone} of $G$. 

\textcolor{black}{Recently, in \cite{agrinsoniproms} we formalized the techniques used by Aubry, McGuire, and Rodier \cite{amw2010}; Ferard \cite{ferard2017}; and Delgado and Janwa \cite{delgadojanwatranversal, delgado2016, moisesjanwa2018congressnueratum, delgado1mod4,   moisesjanwa2017congressnueratum, delgadojanwagoldcase} to factor a polynomial in which the leading form is square-free. In this formalism, we introduce the concept of the degree-gap of a polynomial as follows. 
}

\begin{definition}\label{defn: degreegap}
 Let $F(\textbf{X}) \in \ff[X_1, \dots, X_n]$ be a polynomial of degree $d$ with at least two terms. We define the \textbf{degree-gap} $\gamma(F)$ as the difference $d-d_1$ from the decomposition in Equation \eqref{eqn:graded}. If $F(\textbf{X})$ is a homogeneous polynomial, then $\gamma(F)$ is defined as infinity.
\end{definition}

\textcolor{black}{
\begin{theorem}[Agrinsoni, Janwa \& Delgado  \cite{agrinsoniproms}]\label{Thm: degree-gap}
   Let $F(\textbf{X})= F_d(\textbf{X}) + H(\textbf{X}) \in \ff[\X]$. If $F_d(\textbf{X})$ is square-free, then every factor of $F(\textbf{X})$ has degree-gap at least  $\gamma(F(\textbf{X}))$.   
\end{theorem}
\begin{lemma}[Agrinsoni, Janwa \& Delgado \cite{agrinsoniproms}]\label{cor: degreegap bound on the number of factors}
Let $F(\textbf{X}) = F_d(\textbf{X}) + H(\textbf{X}) \in \ff[\X]$.  If $F_d(\textbf{X})$ is square-free and $(F_d, H) =1$, then $F(\textbf{X})$ has at most $\floor*{\frac{\deg(F)}{\gamma(F)}}$ factors. 
\end{lemma}
}

\textcolor{black}{In section \ref{sect: consequences of the main theorem}, we generalize the concept of the degree-gap of a multivariate polynomial, by introducing the $i^{th}$-degree-gap of the polynomial. Moreover, in Theorem \ref{thm: generalization of extending the degree-gap} we give sufficient conditions to guarantee that the factors of a polynomial have at least one of the gaps of the same number as the $i^{th}$-degree-gap of the original polynomial. This theorem generalized Theorem \ref{Thm: degree-gap}.
\textcolor{black}{As a direct consequence of this result,  we generalize Lemma \ref{cor: degreegap bound on the number of factors}, which bounds the number of factors of a polynomial with a square-free leading form and without homogeneous factors. }
}

A result on the absolute irreducibility of a special class of trinomials was given by
Beelen and Pellikaan \cite{beelen2000newton}. \textcolor{black}{Recently we \textcolor{black}{investigated} the absolute irreducibility of multivariate polynomials with few forms with square-free leading form see propositions (\ref{prop: 4}, \ref{prop: 5}, and \ref{thm: quadrinomials}). In this article, in Theorem \ref{thm: heptanomials} we give a general criterion to test absolute irreducibility. The propositions below (our earlier results (\ref{prop: 4}, \ref{prop: 5}, and \ref{thm: quadrinomials})) are now a direct consequence of our theorem. }

\textcolor{black}{
\begin{proposition}[Agrinsoni, Janwa \& Delgado \cite{agrinsoniproms}]\label{prop: 4}
Let $F(\textbf{X}) = F_d(\textbf{X}) + F_m(\textbf{X}) \in \ff[X_1, \dots, X_n]$, where $\deg(F) = d$. If $F_d(\textbf{X})$ is square-free and $(F_d, F_m) =1$, then $F(\textbf{X})$ is absolutely irreducible. 
\end{proposition}
}

This proposition provides a complete characterization of generalized binomials when the leading form is square-free.

\textcolor{black}{
\begin{proposition}[Agrinsoni, Janwa \& Delgado \cite{AJDabsirred}]\label{prop: 5}
Let $F(X) = F_d(X) + F_{d-e}(X) + F_a(X) \in \ff[X]$, where $d-e >a$. If $F_d(X)$ is square-free,  $(F_d, $\\ $  F_{d-e}, F_a) =1$, and  $a \neq d- me $ for some $m \in \mathbb{Z}$, then $F(X)$ is absolutely irreducible. 
\end{proposition}
}

\textcolor{black}{
\begin{proposition}[Agrinsoni, Janwa \& Delgado \cite{AJDquadrinomials}]\label{thm: quadrinomials}
    Let $F(\X) = F_d(\X) + F_{d-\gamma_1}(\X) + F_{d-\gamma_2}(\X) + F_{d-\gamma_3}(\X) \in \ff[\X]$. Let $F_d(\X)$ be square-free and $(F_d , F_{d-\gamma_1}, $ $ F_{d-\gamma_2} , F_{d-\gamma_3}) =1$. If $\gamma_3 \notin \text{span}_{\mathbb{N}} \{ \gamma_1, \gamma_2\}$, then $F(\X)$  is absolutely irreducible.  
\end{proposition}
}

 In this article, we show that many more multivariate polynomials with square-free leading form are absolutely irreducible. The following result shows that our criterion is valid for almost alll multivariate polynomials.

\begin{proposition}[Gao and Lauder \cite{gao2001square-free}]\label{proposition: Gao-Lauder}
 Most multivariate polynomials are square-free. 
\end{proposition}

\section{New results}\label{sect: New results}

For the rest of this article, we assume that $0 \in \n$.

\begin{theorem}\label{thm: heptanomials}
    Let $F(\X) = F_d(\X) + F_{d-\gamma_1}(\X) + \dots +F_{d-\gamma_m}(\X) \in \ff[\X]$. Let $F_d(\X)$ be square-free and $(F_d , F_{d-\gamma_1} , \dots ,F_{d-\gamma_m}) =1$. If $\gamma_m \notin \spann$, then $F(\X)$  is absolutely irreducible.  
\end{theorem}

\begin{proof}
    Assume \textcolor{black}{by way contradiction} that 

\begin{equation}\label{eqn: factorization pq}
    F = (P_s + P_{s-1} + \dots + P_0)(Q_t + Q_{t-1}+ \dots + Q_0),
\end{equation}
where $P_{i}$ (resp. $Q_j$) is either a form of degree $i$ (resp. $j$) or zero. Now  
\begin{equation}\label{eqn: first cone quadrinomial}
 F_d = P_sQ_t.
\end{equation}

{\bf Let }$k \notin \spann$ {\bf and $0 < k < \gamma_m$.} 

Equating the terms of degree $d-k$, in the expansion of Equation \eqref{eqn: factorization pq}, we obtain the following.
\begin{equation*}
    F_{d-k} = 0 = \sum_{i=0}^k P_{s-i}Q_{t-k + i} =
\end{equation*}
 \begin{equation} \label{eqn: term of degree d-k}
     P_sQ_{t-k} + \sum_{\mathclap{i \notin \spann, \atop 0< i<k}} P_{s-i}Q_{t-k+i} + P_{s-k}Q_t + \sum_{\mathclap{i \in \spann, \atop 0< i<k}} P_{s-i}Q_{t-k+i}.
 \end{equation}
\textbf{Claim:} Equation \eqref{eqn: term of degree d-k} can be rewritten as 
\begin{equation}\label{eqn: term of degree d-k refinement}
     F_{d-k} = 0 = P_sQ_{t-k} + P_{s-k}Q_t,
\end{equation}
for all $0 < k < \gamma_m$. 

For the proof of the claim, it is enough to show that

\begin{equation}\label{eqn: first sum}
    \sum_{\mathclap{i \notin \spann, \atop 0< i<k}} P_{s-i}Q_{t-k+i} = 0,
\end{equation}
and 
\begin{equation}\label{eqn: second sum}
    \sum_{\mathclap{i \in \spann, \atop 0< i<k}} P_{s-i}Q_{t-k+i}=0.
\end{equation}
\vspace{.1in}

\textbf{Part I:} {\it
\textcolor{black}{
For the sake of contradiction, assume that equality in  Equation \eqref{eqn: first sum} does not hold, 
i.e., it 
is  $\neq 0$.}
}

\vspace{.1in}
Let $A = \{ j \mid P_{s-j} \neq 0, j \notin \spann, \ 0<j<k\}$. By our assumption, $A \neq \emptyset$. 
Then, according to the well-ordering principle, there exists a minimal element $a \in A$. 
Equating the terms of degree $d-a$, we obtain:
\begin{equation*}
    F_{d-a} = 0 = \sum_{i=0}^a P_{s-i}Q_{t-a+i} =
\end{equation*}
\begin{equation}\label{eqn: term of degree d-a}
    P_sQ_{t-a} + \sum_{\mathclap{i \notin \spann,\atop 0< i <a }} P_{s-i} Q_{t-a+i} +  P_{s-a}Q_t +\sum_{\mathclap{i \in \spann,\atop 0< i <a }} P_{s-i} Q_{t-a+i} . 
\end{equation}
By the minimality of $a$, $\sum_{i \notin \spann, \ 0< i <a } P_{s-i} Q_{t-a+i} = 0$, since $P_{s-i} = 0$ for all $i \in \{ j \mid j \notin \spann, \  0< j <a\}$. 

\textbf{Claim:}
\begin{equation}\label{eqn: second sum of the term d-a}
\sum_{\mathclap{i \in \spann},\atop 0< i <a } P_{s-i} Q_{t-a+i} = 0.    
\end{equation}

Assume, for the sake of contradiction, that the equality in Equation \eqref{eqn: second sum of the term d-a} does not hold.
Let $B = \{ j \mid Q_{t-a+j} \neq 0, j \in \spann \ 0 < j<a\}$. By assumption $B \neq \emptyset$. Then, since $B$ is a finite set, there exists a maximal element $b \in B$. Note that $a-b \notin \spann$. Equating the terms of degree $d-a+b$, we obtain:
\begin{equation*}
     F_{d-a+b} = 0 = \sum_{i=0}^{a-b} P_{s-i}Q_{t-a+b+i} =
\end{equation*}
\begin{equation}\label{eqn: term of degree d-b}
    P_sQ_{t-a+b} + \sum_{\mathclap{i \notin \spann,\atop 0< i <a-b }} P_{s-i} Q_{t-a+b+i} + P_{s-a+b}Q_t + \sum_{\mathclap{i \in \spann,\atop 0< i <a-b}} P_{s-i} Q_{t-a+b+i}. 
\end{equation}
Observe that 
\begin{enumerate}
    \item $\sum_{i \notin \spann,\ 0< i <a-b } P_{s-i} Q_{t-b+i} =0$, as $P_{s-i} = 0$ for all $i \in \{j \mid j \notin \spann,\ 0< j <a-b\}$, by the minimality of $a$, 
    \item $\sum_{i \in \spann, \ 0< i <a-b } P_{s-i} Q_{t-a+b+i} =0$, as $Q_{t-a+b+i} = 0$ for all $i \in \{ j \mid j  \in \spann,\  0< j <a-b \}$ by the maximality of $b$. 
\end{enumerate}
Therefore, Equation \eqref{eqn: term of degree d-b} can be rewritten as follows
\begin{equation*}
    F_{d-a+b} = 0 = P_sQ_{t-a+b} + P_{s-a+b}Q_t.
\end{equation*}
This implies $P_{s}Q_{t-a+b} = -P_{s-a+b}Q_t$. Hence, $P_s \mid P_{s-a+b}Q_t$ and $Q_t \mid P_sQ_{t-a+b}$. By Equation \eqref{eqn: first cone quadrinomial} and the square-freeness of  $F_d(\textbf{X})$, we have $(P_s, Q_t)=1$.  Therefore, two  divisibilities imply that $P_{s} \mid P_{s-a+b}$, and $Q_t \mid Q_{t-a+b}$. Thus, $P_{s-a+b} = Q_{t-a+b}=0$. This is a contradiction. Therefore, Equation \eqref{eqn: second sum of the term d-a} is zero.   

Now Equation \eqref{eqn: term of degree d-a} can be rewritten as
\begin{equation*}
    F_{d-a} = P_sQ_{t-a} + P_{s-a}Q_t.
\end{equation*}
This implies $P_{s}Q_{t-a} = -P_{s-a}Q_t$. Therefore, $P_{s-a} = Q_{t-a} =0$. This is a contradiction. Hence, Equation \eqref{eqn: first sum} is equal to zero. 

\vspace{.1in}

\textbf{Part II.} \textit{Assume, for the sake of contradiction, that the equality in Equation \eqref{eqn: second sum} is not valid.} 

\vspace{.1in}

Let $C = \{ j \mid Q_{t-k+j} \neq 0, j \in \spann, \ 0<j<k\}$. By our assumption $C \neq \emptyset$. Then, since $C$ is a finite set, there exists a maximal element $c \in C$. 
Equating the terms of degree $d- k+c$ we obtain:

\begin{equation*}
     F_{d-k+c} = 0 = \sum_{i=0}^{k-c} P_{s-i}Q_{t-k+c+i} =
\end{equation*}
\begin{equation}\label{eqn: term of degree d-k+c}
    P_sQ_{t-k+c} +  \sum_{\mathclap{i \notin \spann,\atop 0< i <k-c }}  P_{s-i} Q_{t-k+c+i} + P_{s-k+c}Q_t + \sum_{\mathclap{i \in \spann,\atop 0< i <b }} P_{s-i} Q_{t-k+c+i}. 
\end{equation}

Observe that
\begin{enumerate}
    \item $\sum_{i \notin \spann,\ 0< i <k-c } P_{s-i} Q_{t-k+c+i} = 0$ as $P_{s-i}=0$ for all $i \in \{j \mid j \notin \spann,\ 0< j <k-c \}$ by the proof of Equation \eqref{eqn: first sum}, and
    \item  $\sum_{i \in \spann,\ 0< i <b } P_{s-i} Q_{t-k+c+i} = 0$, as $Q_{t-k+i} = 0$ for all $i \in \{j \mid j \in \spann,\ 0< j <b \}$ by the maximality of $c$. 
\end{enumerate}
Therefore, Equation \eqref{eqn: term of degree d-k+c} can be rewritten as 

\begin{equation*}
    F_{d-k+c} = 0 = P_sQ_{t-k+c} + P_{s-k+c}Q_t.
\end{equation*}

This implies $P_{s}Q_{t-k+c} = -P_{s-k+c}Q_t$. Therefore, $P_{s-k+c} = Q_{t-k+c} =0$. This is a contradiction. Hence, Equation \eqref{eqn: second sum} is equal to zero. Therefore, Equation \eqref{eqn: term of degree d-k refinement} is valid. This implies $P_{s}Q_{t-k} = -P_{s-k}Q_t$. 

{\bf Therefore, $P_{s-k} = Q_{t-k} =0$ for all} $k \notin \spann$
and {\bf $0 < k < \gamma_m$.}

Therefore, equating the forms of degree $d-\gamma_m$, we obtain 

\begin{equation*}
 F_{d-\gamma_m} = \sum_{i=0}^{d-\gamma_m} P_{s-i}Q_{t-\gamma_m+i} = 
\end{equation*}
\begin{equation}\label{eqn: trinomial m=k}
P_{s}Q_{t-\gamma_m} +\sum_{\mathclap{i \notin \spann,\atop 0< i < \gamma_m }} P_{s-i}Q_{t-\gamma_m+i} + P_{s-\gamma_m}Q_{t} + \sum_{\mathclap{i \in \spann,\atop 0< i <\gamma_m }}  P_{s-i}Q_{t-\gamma_m+i}  .
\end{equation}
Now $ \sum_{i \notin \spann,\ 0< i < \gamma_m }P_{s-i}Q_{t-\gamma_m+i} =0$, since $P_{s-i} =0$ for every $i \in \{j \mid j \notin \spann,\ 0< j < \gamma_m\}$ and \\ $\sum_{i \in \spann,\ 0< i <\gamma_m }P_{s-i}Q_{t-\gamma_m+i} =0$, since $Q_{t-\gamma_m+i} =0$ for every $i \in \{j \mid j \in \spann,\ 0< j <\gamma_m\}$.  
Therefore, Equation \eqref{eqn: trinomial m=k} can be simplified as 

\begin{equation}\label{eqn: tangent cone trinomial binomial sum}
     F_{d-\gamma_m} = P_{s}Q_{t-\gamma_m} + P_{s-\gamma_m}Q_{t}.
\end{equation}

The tangent cone of $F(\textbf{X})$ is the product of the tangent cone of $P$ and $Q$ (as the lowest degree form of a polynomial is the product of the lowest degree forms of its factors). Therefore, there exists $b',c'$ such that 

\begin{equation}\label{eqn: tangent cone trinomial}
 F_{d-\gamma_m} = P_{s-b'} Q_{t-c'}, 
\end{equation}
and $P_{s-i} = 0$ for $i > b'$ and $Q_{t-j} = 0$ for $j > c'$.\\

Since $F_{d-\gamma_m}(\X)$ is the tangent cone of $F(\textbf{X})$, it is the product of the tangent cones of $P$ and $Q$. Therefore, $P_{s-\gamma_m} =0$ or $Q_{t-\gamma_m} =0$. If $P_{s-\gamma_m} =0$, then by Equation \eqref{eqn: tangent cone trinomial binomial sum}, the tangent cone of $P(\textbf{X})$ is $P_{s}$. This contradicts $(F_d , F_{d-\gamma_1} ,\dots ,F_{d-\gamma_m}) =1$. Similarly, if $Q_{t-\gamma_m} =0$, then by Equation \eqref{eqn: tangent cone trinomial binomial sum}, the tangent cone of $Q(\textbf{X})$ is $Q_{t}$. This contradicts $(F_d , F_{d-\gamma_1} ,\dots,F_{d-\gamma_m}) =1$.
\end{proof}

\textcolor{black}{As a direct consequence of the proof of this theorem, we conclude the following lemmas.
\begin{lemma}
    Let $F(\X) = F_d(\X) + F_{d-\gamma_1}(\X) + \dots +F_{d-\gamma_m}(\X) \in \ff[\X]$. Let $F_d(\X)$ be square-free and $(F_d , F_{d-\gamma_1} , \dots ,F_{d-\gamma_m}) =1$. If $\gamma_i \notin \text{span}_{\n}\{ \gamma_1, \gamma_2, \dots, \gamma_{i-1}\}$ and $F(\X) = (P_s+P_{s-1}+ \dots+P_0)(Q_t+Q_{t-1}+ \dots+ Q_0)$, then $F_{d-\gamma_i} = P_sQ_{t-\gamma_i}+ P_{s-\gamma_i}Q_t$  
\end{lemma}
\begin{lemma}
    Let $F(\X) = F_d(\X) + F_{d-\gamma_1}(\X) + \dots +F_{d-\gamma_m}(\X) \in \ff[\X]$. Let $F_d(\X)$ be square-free and $(F_d , F_{d-\gamma_1} , \dots ,F_{d-\gamma_m}) =1$. If $k \notin \text{span}_{\n}\{ \gamma_1, \gamma_2, \dots, \gamma_{m-1}\}$, $0< k < \gamma_m$, and $F(\X) = (P_s+P_{s-1}+ \dots+P_0)(Q_t+Q_{t-1}+ \dots+ Q_0)$, then $F_{d-k} = P_sQ_{t-k}+ P_{s-k}Q_t$  
\end{lemma}
}

\section{Generalization of the degree-gap concept and an improved bound on the number of absolutely irreducible factors}\label{sect: consequences of the main theorem}
\begin{definition}
    Let $F(\X) = F_d(\X) + F_{d-\gamma_1}(\X) + \dots +F_{d-\gamma_m}(\X) \in \ff[\X]$ and $F_d(\X)$ be square-free. We define the \textbf{$i^{th}$-degree-gap} $\gamma_i(F)$ as $d-\gamma_i$. If $F(\textbf{X})$ contains less than $i$ forms, then $\gamma_i(F)$ is defined as infinity.    
    
\end{definition}

\begin{theorem}\label{thm: generalization of extending the degree-gap}
Let $F(\X) = F_d(\X) + F_{d-\gamma_1} + F_{d-\gamma_2}(\X) + \dots + F_{d-\gamma_m}(\X) \in \ff[\X]$. Let $k \in \{1,2, \dots, m\}$ be such that $\gamma_k \notin \text{span}_{\n}\{\gamma_1, \gamma_2, \dots, \gamma_{k-1}\}$ and $(F_d, F_{d-\gamma_k}) =1$. If $P(\X) \mid F(\X)$, and $F_d(\X)$ is square-free, then $\deg(P) \geq \gamma_k$.     
\end{theorem}
\begin{proof}
    Let $F(\X) = (P_S+ P_{s-1}+ \dots + P_0)(Q_t + Q_{t-1} + \dots + Q_0)$, where $P_{i}$ (resp. $Q_j$) is either a form of degree $i$ (resp.  $j$) or zero. Then 
    \begin{equation*}
        F_d = P_sQ_t
    \end{equation*}
    where $(P_s, Q_t) =1$. By the proof of Theorem \ref{thm: heptanomials} for $i \notin \text{span}_{\n}\{\gamma_1, \dots, \gamma_{k-1} \}$, $i < \gamma_{k}$, we have 
    \begin{equation*}
        F_{d-i} = P_sQ_{t-i}+ P_{s-i}Q_t.
    \end{equation*}
Since $F_{d-i} =0$, 
\begin{equation*}
    P_sQ_{t-k} = P_{s-k}Q_t.
\end{equation*}
Therefore, $P_s \mid P_{s-i}$ and $Q_{t} \mid Q_{t-i}$, hence $P_{s-i} = Q_{t-i} =0$ for all $i \notin \text{span}_{\n}\{\gamma_1, \dots, \gamma_{k-1} \}$, $i < \gamma_{k}$.
\end{proof}

Consider the homogeneous form of degree $d-\gamma_k$. 

\begin{equation*}
    F_{d-\gamma_k}  = \sum_{i=0}^{\gamma_k} P_{s-i}Q_{t-\gamma_k + i} =
\end{equation*}
 \begin{equation} \label{eqn: term of degree d-k general n-gap}
     P_sQ_{t-\gamma_k} + \sum_{\mathclap{i \notin \spannnn, \atop 0< i<\gamma_k}} P_{s-i}Q_{t-\gamma_k+i} + P_{s-\gamma_k}Q_t + \sum_{\mathclap{i \in \spannnn, \atop 0< i<\gamma_k}} P_{s-i}Q_{t-\gamma_k+i}.
 \end{equation}

Observe that from the proof of Theorem \ref{thm: heptanomials}:
\begin{enumerate}
    \item $\sum_{i \notin \spannnn, \atop 0< i<\gamma_k} P_{s-i}Q_{t-k+i} = 0$ as $P_{s-i} = 0$ for all $i \notin \spannnn$, $0<i<\gamma_k$.
    \item $\sum_{i \in \spannnn, \atop 0< i<\gamma_k} P_{s-i}Q_{t-\gamma_k+i} = 0$ as $Q_{t-\gamma_k+i} = 0$ for all $i \in \spannnn,$ $0< i<\gamma_k$. 
\end{enumerate}
Therefore, we can rewrite Equation \eqref{eqn: term of degree d-k general n-gap} as 
\begin{equation}\label{eqn: equation 16}
    F_{d-\gamma_k} = P_sQ_{s-\gamma_k} + P_{s-\gamma_k}Q_t. 
\end{equation}
Since $(F_d, F_{d-\gamma_k})=1$, $Q_{t-\gamma_k} \neq 0$ and $P_{s-\gamma_k} \neq 0$. Therefore, $s-\gamma_k \geq 0$, hence $s= \deg(P) 
\geq \gamma_k$. 
\begin{remark}\label{rmk: the existence of the term gmmmak}
A direct consequence of this theorem is that if $P(X) = P_s + P_{s-1}+ \dots + P_0$ is a factor of $F(X)$, then $P_{s-\gamma_k} \neq 0$, where $k$ is an in Theorem \ref{thm: generalization of extending the degree-gap}. 
\end{remark}
\begin{corollary}
    Let $F(\X) = F_d(\X) + F_{d-\gamma_1}(\X) + \dots + F_{d-\gamma_m}(\X) \in \ff[\X]$. Let $k \in \{1, \dots, m\}$ be the maximum $k$ that meets the conditions of Theorem \ref{thm: generalization of extending the degree-gap}. If $F_d$ is square-free, then $F(\X)$ has at most $\floor*{\frac{\deg(F)}{\gamma_k}}$ factors.
\end{corollary}

\begin{corollary}
    Let $F(\X) = F_d(\X) + F_{d-\gamma_1}(\X) + \dots + F_{d-\gamma_m}(\X)\in \ff[\X]$. Let $k \in \{1, \dots, m\}$ be the maximum $k$ that meets the conditions of Theorem \ref{thm: generalization of extending the degree-gap}. If $F_d$ is square-free and $\gamma_m < 2\gamma_k$, then $F(\X)$ is absolutely irreducible. 
\end{corollary}
\begin{proof}

Assume that $F(\textbf{X}) = P(\textbf{X}) Q(\textbf{X})$. Then by Remark \ref{rmk: the existence of the term gmmmak}, we have $P_{s-\gamma_k} \neq 0$ and $Q_{t-k} \neq 0$. Therefore, $\deg(T_F) = d-\gamma_m \leq \deg(P_{s-\gamma_k}*Q_{t-\gamma_k}) = s-\gamma_k + t - \gamma_k = d- 2\gamma_k$. Therefore, $\gamma_m \geq 2\gamma_k$ a contradiction.

\end{proof}

\section{Conclusion\textcolor{black}{s}}\label{sect: conclusions}

Absolute irreducibility testing of multivariate polynomials is very important in number theory and algebraic geometry.

In this article, we have given criteria for absolute irreducibility that do not require testing for irreducibility in the ground or extension fields, assuming that the
leading form is square-free. We require multivariate GCD computations
and the square-free property. 
In the case of bivariate polynomials, these two properties reduce to an irreducibility test for a single-variable polynomial computation for the leading homogeneous form. Fast algorithms exist for irreducibility testing over finite fields. 
Therefore, we have fast absolute irreducibility criteria for curves over finite fields. Since curves over finite fields have many applications in coding theory, cryptography, and combinatorics, we have given absolute irreducibility criteria that have important applications.

Indeed, we have applied our criterion to demonstrate the absolute irreducibility of APN polynomials and exceptional polynomials. As a consequence, we have made progress in resolving the exceptional APN conjecture in Agrinsoni, Janwa, and Delgado \cite{AJDKasami}.

As we have quoted,  almost all polynomials are known to be
square-free; our absolute irreducibility criterion is valid for almost all multivariate polynomials.
\textcolor{black}{We have shown by way of counterexamples that our criterion is the
best possible.}

Existing criteria, for example the one given by Stepanov \cite{stepanov1972, stepanov1974}, are applicable to very specific types of polynomials
(e.g., the Kummer extensions). The general criterion of Gao is, in general, impractical, as the algorithm requires the indecomposability of polytopes and is known to be NP-complete \cite{gao2001}.

\end{document}